\documentclass{dcds}
\usepackage{amsmath}
  \usepackage{paralist}
  \usepackage{graphics} %% add this and next lines if pictures should be in esp format
  \usepackage{epsfig} %For pictures: screened artwork should be set up with an 85 or 100 line screen
\usepackage{graphicx}  \usepackage{epstopdf}%This is to transfer .eps figure to .pdf figure; please compile your paper using PDFLeTex or PDFTeXify.
 \usepackage[colorlinks=true]{hyperref}
\hypersetup{urlcolor=blue, citecolor=red}

\def\currentvolume{34}
 \def\currentissue{5}
  \def\currentyear{2014}
   \def\currentmonth{May}
    \def\ppages{X--XX}
     \def\DOI{10.3934/dcds.2014.34.xx}

\newtheorem{theorem}{Theorem}[section]
\newtheorem{corollary}{Corollary}

\newtheorem{lemma}[theorem]{Lemma}\newtheorem*{lem}{Lemma}
\newtheorem{proposition}{Proposition}

\theoremstyle{definition}
\newtheorem{definition}[theorem]{Definition}
\newtheorem{remark}{Remark}
\newtheorem{ques}{Question}
\newtheorem{exa}{Example}

\newcommand{\Ac}{\mathcal {A}}
\newcommand{\Lc}{\mathcal {L}}
\newcommand{\Dc}{\mathcal {D}}
\newcommand{\Q}{\mathbb {Q}}
\newcommand{\Sc}{\mathcal {S}}
\newcommand{\Oc}{\mathcal {O}}
\newcommand{\Hc}{\mathcal {H}}
\newcommand{\Pc}{\mathcal {P}}
\newcommand{\Rc}{\mathcal {R}}
\newcommand{\Vc}{\mathcal {V}}
\newcommand{\N}{\mathbb {N}}
\newcommand{\T}{\mathbb {T}}
\newcommand{\Tc}{\mathcal {T}}
\newcommand{\R}{\mathbb {R}}
\newcommand{\Z}{\mathbb {Z}}
\newtheorem*{propo}{Proposition}

\title[invariant submanifolds of symplectic    dynamics]
      {When are the invariant submanifolds of symplectic    dynamics Lagrangian?}

\author[Marie-Claude Arnaud]{}


 \email{Marie-Claude.Arnaud@univ-avignon.fr}

\begin{document}
\maketitle

% Enter the first author's name and address:
\centerline{\scshape Marie-Claude Arnaud}
\medskip
{\footnotesize
% please put the address of the first author
 \centerline{Avignon University, LMA EA 2151}
   \centerline{F-84000, Avignon, France}
    \centerline{Member of the Institut universitaire de France}
   \centerline {supported by ANR-12-BLAN-WKBHJ}
} % Do not forget to end the {\footnotesize by the sign }

\bigskip

% The name of the associate editor will be entered by an editorial staff
% "Communicated by the associate editor name" is not needed for special issue.
 \centerline{(Communicated by Kuo-Chang Chen)}

%The abstract of your paper
\begin{abstract}
Let $\Lc$ be a  $D$-dimensional submanifold   of a $2D$ dimensional exact symplectic manifold  $(M, \omega)$ and let $f: M\rightarrow M$ be a symplectic diffeomorphism. In this article, we deal with the link between the dynamics $f_{|\Lc}$ restricted to $\Lc$ and the geometry of $\Lc$ (is $\Lc$ Lagrangian, is it smooth, is it a graph\dots?).

 We prove different kinds of results.
\begin{enumerate}
\item for $D=3$, we prove that is $\Lc$ if a torus that carries some characteristic loop, then either $\Lc$ is Lagrangian or $f_{|\Lc}$ can not be  minimal (i.e. all the orbits are dense) with $(f^k_{|\Lc})$ equilipschitz;
\item for a Tonelli Hamiltonian of $T^*\T^3$, we give an example of an invariant    submanifold $\Lc$ with no conjugate points  that is not Lagrangian and such that for every $f:T^*\T^3\rightarrow T^*\T^3$ symplectic, if $f(\Lc)=\Lc$, then $\Lc$ is not minimal;
\item with some hypothesis for the restricted dynamics, we prove that some invariant Lipschitz $D$-dimensional submanifolds of Tonelli Hamiltonian flows are in fact Lagrangian, $C^1$ and graphs;
\item we give similar results for $C^1$ submanifolds with weaker dynamical assumptions.
\end{enumerate}

\end{abstract}

%The title of your section 1
\section{Introduction}

When studying smooth symplectic  dynamical systems, we are often led to look for  their invariant submanifolds.  In the symplectic setting, mathematicians generally ask that the invariant submanifold in question is Lagrangian. But why?%\\

One possible reason is the following result due to Michel Herman (see \cite{He1}):
\begin{propo}(M.~Herman) Let $F$ be a   symplectic $C^1$ diffeomorphism of an exact symplectic $2d$-dimensional manifold $(M, \omega)$ and let $\Tc\subset M$ be a $C^1$ invariant $d$-dimensional torus. Assume that the restricted dynamics $f_{|\Tc}$ is $C^1$-conjugated to an ergodic rotation of $\T^d$. Then $\Tc$ is Lagrangian.
\end{propo}
Under some assumptions, the K.A.M. theorem (for Kolmogorov Arnol'd Moser, see \cite{Bo}) gives the existence of a lot of such invariant tori of  symplectic dynamics.% \\

But there may exist other invariant manifolds that are not K.A.M. tori.  The simplest example is when you consider the identity map: of course every submanifold is invariant... %\\

Moreover, observe that the set of $d$-dimensional $C^1$ Lagrangian tori in a $2d$-dimensional symplectic manifold has no interior in the set of  $d$-dimensional $C^1$ tori endowed with the $C^1$ topology. Hence the set of $C^1$ Lagrangian submanifolds is small.%\\

Observe too that Lagrangian submanifolds are very flexible from the point of view of symplectic dynamics:
 consider a Lagrangian submanifold $\Tc$ of $\T^d\times \R^d$ that is homotopic to the zero section, let $g: \Tc\rightarrow \Tc$ be any diffeomorphism of $\Tc$ that is homotopic to identity. Using generating functions (see for example \cite{AA}), it is easy to extend $g$ to a symplectic diffeomorphism of $\T^d\times \R^d$: this will be done in appendix \ref{s4}. Hence the Lagrangian property allows all the possible dynamics on the submanifold: indeed, you just ask that $g$ preserves the vanishing $2$-form. If your submanifold is not Lagrangian, then your condition ``preserving the symplectic form'' is not trivial and you can hope to avoid certain dynamics, as minimal dynamics. In \cite{He1}, M.~Herman asked the following question ($T^*\T^d=\T^d\times \R^d$ is endowed with its usual symplectic form):

\begin{ques}\label{Q1} (M.~Herman) Let $F$ be a symplectic diffeomorphism of $\T^d\times \R^d
 $ that is homotopic to  Id  and let $\Tc$ be a $C^1$ torus that is:
\begin{itemize}
\item invariant by $F$;
\item  homotopic to $\{ r=0\}$;
\item such that the restricted dynamics $F_{|\Tc}$ is minimal (i.e. all its orbits are dense in $\Tc$).
\end{itemize}
 Is the torus necessarily Lagrangian?
\end{ques}
The answer to this question is yes for $d=1, 2$, but unknown for higher dimensions.%\\

Let us think about higher dimensions. We assume that $\Lc$ is a closed and without boundary  $n$-dimensional submanifold of a $2n$-dimensional exact symplectic manifold  $(M, \omega=d\alpha)$ and that $f$ is a symplectic diffeomorphism of $M$ such that $\Lc$ is $f$-invariant and $f_{|\Lc}$ is minimal. Then at every $x\in\Lc$, $F(x)=\ker \omega_{|T_x\Lc}$ is a linear subspace of $T_x\Lc$. We denote its dimension by $p(x)$. Then $n-p$ is even, $p$ is invariant by $f$ because $f$ is symplectic and $p$ is lower semi-continuous. This implies that $p$ is constant on $\Lc$. If $n=p$, then $\Lc$ is Lagrangian. Then let us assume that $n-p=2m>0$. The following lemma, whose proof is very simple,  shows that there are some restrictions for the  characteristic bundle $\ker\omega_{|T\Lc}$:
 \begin{lem}
 Let $\Lc$ be a closed   and without boundary n-dimensional  submanifold of a 2n-dimensional exact symplectic manifold $(M,\omega=d\alpha)$. We assume that $\ker\omega_{|T\Lc}$ defines a $(n-2m)$-dimensional bundle along $\Lc$. Then no closed and without boundary submanifold  in $\Lc$ is transverse to this $(n-2m)$-dimensional bundle.
 \end{lem}
 If such a submanifold $N$ exists, then $(N, \omega)$ is exact symplectic, closed and without boundary. By Stoke's theorem, this is impossible:
$$ \int_N\omega^{\wedge m}=\int_Nd\alpha\wedge (\omega^{\wedge (m-1)})=\int_Nd(\alpha\wedge \omega^{\wedge (m-1)})=\int_{\partial N=\emptyset}\alpha\wedge \omega^{\wedge (m-1)}=0.$$
There is a case where it is easy to build such a transverse section to the characteristic subbundle:
\begin{corollary}\label{Co0}
 Let $\Tc$ be an embedded $3$-dimensional torus  of a 6-dimensional exact symplectic manifold $(M,\omega=d\alpha)$. Then the following assertion is impossible:%\\\

  $\ker\omega_{|T\Lc}$ defines a $1$-dimensional bundle in the tangent space of  $\Tc$ and the corresponding  characteristic leaves
 define a   locally trivial fibration in circles of $\Tc$.

\end{corollary}
 We deduce a corollary in the spirit of Herman's question:
 \begin{corollary}\label{Co1}
 Let   $\Tc\subset M$ be a $3$-dimensional $C^2$-embedded torus  of a $6$-dimensional  exact symplectic manifold $M$ such that
  there exists a $C^1$-embedded characteristic loop $j:\T \rightarrow \Tc$ in $\Tc$, i.e.  such that:
$$\forall t \in\T , \gamma'(t)\in\ker\omega_{|T_{\gamma(t)}\Tc}.$$
Then,
\begin{itemize}
\item either $\Tc$ is Lagrangian;
\item or for every symplectic diffeomorphism $f$ of $M$ such that $f(\Tc)=\Tc$ and the family $(f^k_{|\Tc})_{k\in\Z}$ is equilipschitz, the restricted dynamics $f_{|\Tc}$ is not minimal.
\end{itemize}
\end{corollary}

\begin{remark}\begin{enumerate}
\item The class of submanifolds $\Tc$ that are described in corollary \ref{Co1} contains submanifolds that clearly cannot carry   a minimal restricted symplectic dynamics, as submanifolds that are non-Lagrangian everywhere but Lagrangian on some open subset. But the example given in proposition \ref{P1} is an example of such a submanifold for which the result is not so   trivial.
\item Unfortunately, we need the assumption that ``$(f^k_{|\Tc})_{k\in\Z}$ is equilipschitz'' to prove the corollary. This condition is satisfied when $f_{|\Tc}$ is Lipschitz conjugated to a  rotation (that is a priori weaker than the condition ``$C^1$ conjugated to an ergodic  rotation'' as in Herman's proposition).
\item M.~Herman proved in \cite{He2} that any diffeomorphism $f$ of $\T^3$ that is homotopic to identity and such that the family $(f^k)_{k\in\Z}$ is equi-continuous has a unique rotation number; he proved too that $f$    is $C^0$ conjugated to an ergodic rotation when this rotation number corresponds to an ergodic rotation.
\end{enumerate}
\end{remark}

 Question \ref{Q1} concerns the link between the dynamics and the Lagrangian property.

 In the remaining of this article, we will consider this kind of questions for a wide class of symplectic dynamics, the ones that correspond to   Tonelli Hamiltonians:

\begin{definition}
Let $M$ be a $d$-dimensional closed manifold and let us denote   its cotangent bundle endowed with its usual symplectic form  by  $(T^*M, \omega)$.

 A $C^2$ function $H: T^*M\rightarrow \R$ is called a {\em Tonelli Hamiltonian} if it is:%\\
\begin{itemize}
\item superlinear in the fiber, i.e. $$\forall A\in \R, \exists B\in \R, \forall (q,p)\in T^*M, \| p\|\geq B\Rightarrow  H(q,p) \geq A\| p\|;$$
\item $C^2$-convex in the fiber i.e. for every $(q,p)\in T^*M$, the Hessian $\frac{\partial^2H}{\partial p^2}$ of $H$ in the fiber direction is positive definite as a quadratic form.
\end{itemize}
We denote the Hamiltonian flow of $H$ by $(\varphi_t)$ and the Hamiltonian vector-field by $X_H$.

\end{definition}
It is easy to build a Tonelli Hamiltonian having a non-Lagrangian invariant graph. The following example was built in my thesis and is commented in \cite{He1}:

\begin{exa}Let us consider the Tonelli  Hamiltonian $H:\T^2\times \R^2\rightarrow \R$ defined by $H(\theta_1, \theta_2, r_1, r_2)=\frac{1}{2}(r_1-\psi (\theta_2))^2+\frac{1}{2}r_2^2$ where $\psi:\T\rightarrow \R$ is a non-constant function.
\medskip

Then the torus $\{(\theta_1, \theta_2, \psi(\theta_2), 0); (\theta_1, \theta_2)\in\T^2\}$ is a non-Lagrangian torus that is invariant by the Hamiltonian flow of $H$. Moreover, the restricted dynamics is conjugated to a non-ergodic rotation of $\T^2$ (the identity map) and normally elliptic.
\end{exa}

As noticed by M.~Herman, this counter-example contains orbits that are {\em non-minimizing} where:

\begin{definition}
Let $H:T^*M\rightarrow \R$ be a Tonelli Hamiltonian and  let $L: TM\rightarrow \R$ be  the Lagrangian that is associated to  $H$. It is defined by $$\displaystyle{L(q, v)=\max_{p\in T^*_qM} (p.v-H(q,p))}.$$

The Lagrangian action $A_L(\gamma )$ of a $C^1$ arc $\gamma: [a, b]\rightarrow M$ is defined by:
$$A_L(\gamma )=\int_a^bL(\gamma (s), \dot\gamma (s))ds.$$
\medskip

An orbit $(q_t, p_t)_{t\in\R}$ is {\em minimizing} (resp. {\em locally minimizing}) if for every $a<b$ in $\R$, the arc $(q(t))_{t\in[a, b]}$ minimizes (resp. locally minimizes) the action among all the $C^1$ arcs $\gamma: [a, b]\rightarrow M$ such that $\gamma(a)=q(a)$ and $\gamma(b)=q(b)$.
\end{definition}

\begin{remark}
\begin{enumerate}
\item It is well-known that an orbit is locally minimizing if and only if it has no conjugate vectors.
\item A classical result asserts that any orbit of a Tonelli Hamiltonian flow that is contained in an invariant Lipschitz Lagrangian graph is locally minimizing.% \\
\end{enumerate}
\end{remark}
 Then we have a new  question:

 \begin{ques}
 Is a   $d$-dimensional submanifold with no conjugate points that is invariant by a Tonelli Hamiltonian flow necessarily  Lagrangian?
 \end{ques}
We  give a negative answer for the geodesic flow of the flat metric of $\T^3$:
\begin{proposition}\label{P1}
Let $H:\T^3\times \R^3\rightarrow \R$ be the Hamiltonian defined by:
$$H(\theta, r)=\frac{1}{2}(r_1^2+r_2^2+r_3^2) $$
and let $j:\T^3\rightarrow \T^3\times\R^3$ be the embedding map defined by:
$$j(\theta)=(\theta; \cos 2\pi\theta_3,\sin 2\pi\theta_3, 0).$$
Then the submanifold $\Tc=j(\T^3)$ is invariant by the Hamiltonian flow of $H$, non-Lagrangian and with no conjugate points.%\\

We can modify this example in such a way that the torus has no conjugate points and is non-Lagrangian and not a graph, isotopic or non isotopic to $\T^3\times \{ 0\}$.

Moreover, for every symplectic diffeomorphism $f:T^* \T^3 \rightarrow T^*\T^3$ satisfying $f(\Tc)=\Tc$, then $f_{|\Tc} $ is not minimal (i.e. has at least one non-dense orbit).

\end{proposition}

  \begin{remark} \begin{enumerate}
  \item The counter-example $\Tc$ is foliated by invariant isotropic $2$-dimen\-sional tori on which the Hamiltonian  flow of $H$ is conjugated to a rotation flow.
  \item The characteristic subbundle of $T_x\Tc$ is 1-dimensional and have some compact 1-dimensional leaves, hence we are in  a case where we can apply corollary \ref{Co1}, but we improve the result contained in corollary \ref{Co1} for this example, because we don't assume that the family $(f^n)$ is equilipschitz.

  \end{enumerate}
  \end{remark}

  Let us now give some dynamical conditions that imply that an invariant manifold with no conjugate points  is Lagrangian.
\begin{theorem}\label{T1}  Let $\Lc$ be a  Lipschitz  invariant $d$-dimensional submanifold of a Tonelli Hamiltonian flow   $(\varphi_t)$ of $T^*M$ such that: \begin{itemize}
\item  there exist two sequences $(t_n)$ and $(s_n)$  tending to $+\infty$ such that the families $(\varphi_{t_n|\Lc})_{n\in\N}$ and $(\varphi_{-s_n|\Lc})_{n\in\N}$ are equilipschitz;
\item all the orbits that are contained in $\Lc$ have no conjugate points.
\end{itemize}
Then $\Lc$ is Lagrangian and is the graph of a $C^1$ function.
\end{theorem}
 \begin{corollary}\label{C1}
  Let $\Tc$ be a  Lipschitz  invariant $d$-dimensional torus of Tonelli Hamiltonian flow   of $T^*\T^d$ such that: \begin{itemize}
\item  the time $T$ flow restricted to $\Tc$ is Lipschitz conjugated to a (not necessarily ergodic) rotation for one $T>0$ ;
\item all the orbits that are contained in $\Tc$ have no conjugate points.
\end{itemize}
Then $\Tc$ is Lagrangian and is the graph of a $C^1$ function.
 \end{corollary}
\begin{remark}\begin{enumerate}
\item There exist examples of Tonelli flows that satisfy the hypothesis of theorem \ref{T1} and for which the restricted dynamics is not minimal. Hence the question that we answer is different from Herman's question.
 \item Let us notice that we obtain that $\Lc$ is more than just Lagrangian: it is a graph. Hence we prove that with some assumption on the dynamics, the invariant manifold is a graph. This kind of result is what is generally called a ``multidimensional Birkhoff theorem''. %\\

The only previous known results concerned Lagrangian submanifolds:  first results  are due to M.~Herman (see \cite{He1}) and M.~Bialy and L.~Polterovich (see \cite{Bia1, BiaPol1, BiaPol2, BiaPol3}) by assuming that the restricted dynamics is chain recurrent. Without dynamical assumption but assuming again  that    the  invariant manifold is Lagrangian, I improved the result in \cite{Ar2}. Then P.~Bernard and J.~dos Santos proved a similar result for invariant Lipschitz Lagrangian submanifolds (see \cite{BeSa});
\item We prove too a result of regularity in the previous statements: assuming that the invariant manifold is Lipschitz, we conclude that it is $C^1$. Similar results were proved in \cite{Ar1} for Lipschitz invariant Lagrangian graphs.
\end{enumerate}
 \end{remark}
 If we know that the invariant submanifold is $C^1$, we can improve the statements.

 \begin{theorem}\label{T2}
 Let $M$ be a closed $n$-dimensional manifold and let $N\subset T^*M$ be a closed $n$-dimensional $C^1$ submanifold with no conjugate points  that is invariant by some Tonelli flow $(\varphi_t)$. We assume that for a dense subset $D\subset N$ : %\\
 $$\forall x\in D, \forall v\in T_xN, \min\{ \liminf_{t\rightarrow +\infty}\| D\varphi_tv\|, \liminf_{t\rightarrow -\infty} \| D\varphi_tv\|\}<+\infty.$$

  Then $N$ is a   Lagrangian submanifold of $T^*M$ that is a graph.
 \end{theorem}
 \begin{remark}
Two kinds of general flows of $N$ satisfy these hypotheses: the gradient flows and the flows coming from the action of a compact Lie group. We give two examples of corollaries.
\end{remark}
\begin{corollary}\label{C4}
   Let $M$ be a closed $n$-dimensional manifold and let $N\subset T^*M$ be a closed $n$-dimensional $C^1$ submanifold with no conjugate points  that is invariant by some Tonelli flow $(\varphi_t)$. We assume that there exists a compact $C^1$ Lie group $G$ that acts on $N$ and that contains the diffeomorphisms  $ \varphi_{t|N} $ for $t\in\R$.  Then $N$ is a   Lagrangian submanifold of $T^*M$ that is a graph.
     \end{corollary}

       If we just ask that the manifold is Lipschitz, then the statement is false as explained by the following example.

     \begin{exa} We consider the following Tonelli Hamiltonian defined on $T^*\T^3=\T^3\times \R^3$.
     $$H(\theta_1, \theta_2, \theta_3; r)=\frac{1}{2}\| r\|^2+\cos(4\pi\theta_3).$$
     We check easily that the set:
     $$N=\T\times \T\times \{ (\theta_3, \pm\sqrt{2}\sqrt{1-\cos(4\pi\theta_3)}); \theta_3\in [0, \frac{1}{2}]\}\times \{0\}\times\{ 0\}$$
     is invariant by the Hamiltonian flow, with no conjugate  points, Lipschitz Lagrangian, such that every
    point $x$ of $N$:
      $$  \forall v\in T_xN, \min\{ \liminf_{t\rightarrow +\infty}\| D\varphi_tv\|, \liminf_{t\rightarrow -\infty} \| D\varphi_tv\|\}<+\infty,$$
      but $N$ is not a graph\dots Let us notice that $N$ is not homotopic to the zero section.
     \end{exa}

      \begin{corollary}\label{C3}
   Let $M$ be a closed $n$-dimensional manifold and let $N\subset T^*M$ be a closed $n$-dimensional $C^1$ submanifold with no conjugate points  that is invariant by some Tonelli flow $(\varphi_t)$. We assume that $(\varphi_{t|N})$ is such that:
   \begin{enumerate}
   \item[$\bullet$] the   non-wandering set for $(\varphi_{t|N})$ is a finite union of periodic orbits or critical points;
   \item[$\bullet$] all the periodic orbits of $(\varphi_{t|N})$ are non-degenerate in the following sense: if $\tau$ is the period of such a period point,  the multiplicity of $1$  as an eigenvalue   of $D\varphi_{\tau|TN}$ is   one and $-1$ is not an eigenvalue; for the critical points, we assume the Hamiltonian is Morse at the critical points contained in $N$.
   \end{enumerate}
 Then $N$ is a   Lagrangian submanifold of $T^*M$ that is a graph.
     \end{corollary}

 \subsection{Structure of the article}
Section 2 contain the proof of the results that are in the spirit of Herman's question: proposition \ref{P1} and corollaries \ref{Co0} and  \ref{Co1}.

In section 3, we prove that with some hypotheses on the restricted dynamics, some Lipschitz manifolds that are invariant by a Tonelli Hamiltonian flow are in fact $C^1$, Lagrangian and graphs.

In section 4, we prove similar results with a less restrictive dynamical hypothesis but by assuming that the invariant manifold is $C^1$ .

\section{Proof of  proposition \ref{P1} and corollaries \ref{Co0} and  \ref{Co1}}
\subsection{Proof of corollary \ref{Co0}}
 Let $\Tc$ be an embedded $3$-dimensional torus  of a 6-dimensional exact symplectic manifold $(M,\omega=d\alpha)$. Assume that
  $\ker\omega_{|T\Lc}$ defines a $1$-dimensional bundle along $\Tc$ and the corresponding  characteristic leaves
 define a   locally trivial fibration in circles of $\Tc$.  We denote by $\Sc$ the quotient manifold. Then $\Sc$ is a closed surface. This surface is endowed with the quotient symplectic form $\Omega$, then it is orientable.

 Then we have a Serre fibration and the exact homotopy sequence is:
 $$
\ldots \rightarrow \pi_2(\T )\rightarrow \pi_2(\T^3)\rightarrow  \pi_2(\Sc)\rightarrow \pi_1(\T)\rightarrow \pi_1(\T^3)\rightarrow \pi_1(\Sc)\rightarrow \pi_0(\T)=\{0\}.
$$
Hence $\pi_1(\Sc)$ is Abelian: $\Sc$ is the sphere or the torus. But $\Sc$ cannot be the sphere because there is no surjection from $\Z=\pi_1(\T)$ to $\Z^3=\pi_1(\T^3)$. Moreover, the arrow $\pi_1(\T)\rightarrow \pi_1(\T^3)$ corresponds to an inclusion map and we have the exact sequence $\pi_1(\T)\rightarrow\pi_1(\T^3)\rightarrow \pi_1(\Sc)=\Z^2$; because there is no injective morphism from $\Z^3=\pi_1(\T^3)$ to $\pi_1(\Sc)=\Z^2$,  the inclusion map cannot be zero and then the fiber of the bundle is not homotopic to a point.

We can then build a section of this bundle. To do that, we use a vectofield $X$ of $\T^3$  whose flow $(\varphi_t)$ is $1$-periodic and whose orbits are the leaves of the previous bundle. Such a vector field exists because the orientation of $\T^3$ and of $\Sc$ give an orientation of the leaves.  We lift $(\varphi_t)$ into the flow $(\tilde\varphi_t)$ of $\R^3$. As the leaves of the bundle are not homotopic to a point, there exists $v\in \Z^3\backslash\{ 0\}$ such that $\tilde\varphi_1-{\rm Id}=v$. Using an isomorphism of $\Z^3$, we can assume that $v=\lambda e_1$ where $e_1$ is the first vector of the canonical base.
We define a function $u:\R^3\rightarrow \R$ by: $u(x)=\int_0^1<\tilde \varphi_t(x), e_1>dt$. Then:
\begin{equation*}\begin{split}du(x).X(x)&=\frac{d}{dt}\int_0^1<\tilde \varphi_{s+t}(x), e_1>ds\\
 =&\frac{d}{dt}\int_t^{1+t}<\tilde \varphi_s(x), e_1>ds=<\tilde\varphi_1(x)-x,e_1>=|\lambda |\end{split}\end{equation*}
We deduce that $\Hc=\{ u=0\}$ is a   surface that meets every orbit of $(\tilde \varphi_t)$ at exactly one point. Moreover, if $h=(h_1, h_2, h_3)\in \Z^3$, then: $u(x+h)=u(x)+h_1$. Then $\Hc$ is a retract of $\R^3$ and is then connected. Moreover, if $u([0, 1]^3)=[m, M]$, then for every $x\in \R^3$ such that $u(x)=0$, we can write (here $[x]$ denote the vector whose components are the integer parts of the components of $x$)
$0=u(x)=u(x-[x])+[x_1]$ and deduce $[x_1]\in [-M, -m]$. This fact and the fact that $\Hc$ is invariant by the translations with vector in $\{ 0\}\times \Z^2$ implies that the projection of $\Hc$ on $\T^3$ is a compact surface. It is then a section of the bundle.

 \subsection{Proof of corollary \ref{Co1}}

We assume   that $\Tc\subset M$ is a $3$-dimensional $C^2$ submanifold such that:
\begin{itemize}
\item $\Tc$ is not Lagrangian;
\item $\Tc$ is a  $C^2$-embedded $3$-dimensional torus;
\item $\Tc$ is invariant by a $C^1$ symplectic diffeomorphism $f$;
\item   there exists a $C^1$-embedded characteristic loop $j:\T \rightarrow \Tc$ in $\Tc$, i.e.  such that:
$$\forall t \in\T , \gamma'(t)\in\ker\omega_{|T_{\gamma(t)}\Tc}.$$
\end{itemize}
Let us assume that $f:M\rightarrow M$ is a symplectic diffeomorphism such that $f(\Tc)=\Tc$, $(f^n_{|\Tc})_{n\in\N}$ is equilipschitz  and such that $f_{|\Tc}$ is minimal. We have then noticed that $p=\dim(\ker \omega_{|T_x\Tc})$ is constant along $\Tc$ and such that $n-p$ is even. As $\Tc$ is not Lagrangian, then $p=1$:
$$\forall x\in\Tc, \dim\left( \ker\omega_{|T_{x}\Tc}\right)=1.$$
We choose an orientation on $\Tc$. We denote by $P_0(x)\subset T_x\Tc$ a $2$-plane that is transverse to $\ker\omega_{|T_x\Tc}$ (it may be the orthogonal subspace to $\ker\omega_{|T_x\Tc}$ for some fixed Riemannian metric) and that continuously depends on $x\in\Tc$. Let us now choose $u(x), v(x)\in P_0(x)$ such that $\omega (u(x), v(x))>0$, and let us complete it with $X(x)\in \ker\omega_{|T_x\Tc}$ such that $(u(x), v(x), X(x))$ is oriented and $\|ÊX(x)\|=1$ (for some fixed Riemannian metric). Clearly,     $\ker\omega_{|T_x\Tc}$ and $X$ continuously depend on $x$ (even if $u(x)$ and $v(x)$ may depend on $x$ in a non-continuous way). In fact, the dependence is $C^1$ if $\Tc$ is $C^2$.  Hence we can define the flow $(\varphi_t)$ of $X$.  %\\

As $f$ is symplectic, there exists a $C^1$ map $\lambda:\Tc\rightarrow \R^*$ such that: $\forall x\in \Tc, DfX(x)=\lambda (f(x))X(f(x))$.
We have: $Df^nX(x)=\lambda (f(x)).\lambda(f^2(x))\dots\break \lambda (f^n(x))X(f^n(x))$. As $(f_{|\Tc}^k)_{k\in\Z}$ is uniformly Lipschitz, we deduce that there exist two positive constants $C>c>0$ such that:
$$\forall x\in\Tc, \forall n\in\N, c\leq |\lambda (f(x)).\lambda(f^2(x))\dots \lambda (f^n(x))|\leq C.$$

Let us recall that: $\forall t\in\T, \gamma'(t)\in\ker\omega_{|T_{\gamma(t)}\Tc}$. If we reparametrize $\gamma$ and change its domain of definition, we have: $\gamma'(t)=X(\gamma(t))$, hence $\gamma$ is a periodic orbit of $X$ with fixed period $T>0$; if we use the notation $x_0=\gamma(0)$, we have $\varphi_T(x_0)=x_0$. Then for every $n\in\N$, $f^n(x_0)$ is a periodic orbit for $X$ whose period $T'$ satisfies: $\frac{T}{C}\leq T'\leq \frac{T}{c}$.
Hence every point of   the dense subset $\{ f^nx_0; n\in\Z\}$, is a periodic point for $X$ whose period is between $\frac{T}{C}$ and $\frac{T}{c}$. We deduce that all the points of $\Tc$ are periodic for $(\varphi_t)$ and they have a period between $\frac{T}{C}$ and $\frac{T}{c}$ (this period is not necessarily the minimal one).
\begin{lemma}\label{L1}
$\T$ acts smoothly and freely on $\Tc$ via a reparametrization of $X$.
\end{lemma}
\begin{proof}[Proof of lemma \ref{L1}]
Let us notice that the fact that all the points of $\Lc$ are periodic for $X$ with an upper  bound for the period does not imply that a reparametrization of the flow defines a smooth and free action of $\T$: you can have a period doubling close to a fixed periodic point. The argument that will give the result is the fact that $f$ is minimal.

 Let us denote by $\tau(x)$ the minimal period of $x\in\Tc$  for $X$. As $X$ has no zero  and $\Tc$ is compact, $\tau$ is bounded from below by some positive constant. Hence there exist $0<a<b=\frac{T}{c}$ such that $\forall x\in\Tc, a\leq \tau (x)\leq b$. Moreover, $\tau$ is lower semi-continuous. Let us assume that $\tau$ is not continuous at some $x_0\in\Tc$.  Then there exists a sequence $(y_n)\in\Tc$ such that $\displaystyle{\lim_{n\rightarrow \infty} y_n=x_0}$ and $\displaystyle{\lim_{n\rightarrow \infty}\tau(y_n)>\tau (x_0)}$. Every period of $x_0$ being a multiple of $\tau(x_0)$, we have then: $\displaystyle{\lim_{n\rightarrow \infty}\tau(y_n)\geq 2\tau (x_0)}$.

 Let us notice that the set of such points of discontinuity of $\tau$ is invariant by $f$: the fact that $\displaystyle{\limsup_{y\rightarrow x_0}\tau(y)\geq 2\tau(x_0)}$ means that if you choose a small surface of section $N$ at $x_0$ for $X$, then there exists close to $x_0$ periodic points for the first return map whose period is at least 2.

  Because $f$ is minimal and $\tau$ is lower semi-continuous, there exists an increasing sequence $(k_n)\in\N$ such that $\displaystyle{\lim_{n\rightarrow \infty}f^{k_n}(x_0)=x_0}$ and $\displaystyle{\lim_{n\rightarrow \infty} \tau(f^{k_n}(x_0))\geq 2\tau (x_0)}$. Using this and the fact that the set of the points of discontinuity of $\tau$ is invariant by $f$, we build an increasing  sequence $(m_n)$ of integers such that: $\tau (f^{m_{n+1}}(x_0))\geq \frac{3}{2}\tau(f^{m_n}(x_0))$. This contradicts the fact that all the minimal periods are between $a$ and $b$.

Hence the minimal period $\tau$ continuously depends on the considered point. A classical result (implicit function theorem) then implies that $\tau$ is $C^1$.
\end{proof}
We will denote the reparametrization of $(\varphi_t)$ that describes a free and smooth action of $\T$ by $(\psi_t)$ and the corresponding vector field by $Y$.

We   consider the equivalence   relation  $\Rc$ defined on $\Lc$ by $(\varphi_t)$: $x\Rc y$ if for some $t\in\R$, we have $y=\varphi_t(x)$. Then $P:\Tc\rightarrow \Lc/\Rc$ is a fiber bundle whose base space is  a 2-dimensional   closed manifold denoted by $\Sc$ and whose fiber is $\T$.  In other words, we can apply corollary \ref{Co0} to conclude.

\subsection{Proof of proposition \ref{P1}}
We  consider the embedding $j: \T^3\rightarrow \T^3\times \R^3$ that is defined by:
$$j(\theta_1, \theta_2, \theta_3)= (\theta_1, \theta_2, \theta_3; \cos(2\pi\theta_3), \sin(2\pi\theta_3), 0).$$
The submanifold $\Lc=j(\T^3)$ is invariant by the Hamiltonian flow of $H(\theta;r)=\frac{1}{2}(r_1^2+r_2^2+r_3^2)$ that is the geodesic flow for the flat metric. For the flat metric, all the orbits have no conjugate points, hence $\Lc$ has  no conjugate points.%\\

Let us fix $\theta\in\T^3$. The tangent subspace  $T_{j(\theta)}\Lc$ is generated by $e_1$, $e_3$, $e_3$ where $e_1=(1, 0,0;0,0,0)$, $e_2=(0, 1, 0; 0,0,0)$,
 $e_3= (0, 0, 1; -2\pi\sin(2\pi\theta_3), 2\pi\cos(2\pi\theta_3),\break 0)$. We have:  $\omega(e_1, e_2)=0$ and $\omega (ae_1+be_2, e_3)=2\pi(b\cos(2\pi\theta_3)-a\sin(2\pi\theta_3))$. Hence the kernel of the restriction of $\omega$ to $T_{j(\theta)}\Lc$ is the line $D(\theta)$ generated by the characteristic field $X(j(\theta))=(\cos(2\pi\theta_3), \sin(2\pi\theta_3), 0;0,0,0)$ that is too the Hamiltonian vector field of $H$.

We can integrate this vectorfield $X$ along $\Lc$, its flow $(\psi_t)$ is defined by: $$\psi_t(j(\theta_1, \theta_2, \theta_3))=j(\theta_1+t\cos(2\pi\theta_3), \theta_2+t\sin(2\pi\theta_3), \theta_3)$$ and then is conjugated (via $j$) to the   flow $g_t$ of $\T^3$ defined by :  $g_t(\theta_1, \theta_2, \theta_3)=(\theta_1+t\cos(2\pi\theta_3), \theta_2+t\sin(2\pi\theta_3), \theta_3)$. The foliation of $\T^3$ by the 2-dimensional tori $\Tc_{\theta_3}=\T^2\times\{ \theta_3\}$ is invariant by $(g_t)$.

Moreover, there are two cases:
\begin{itemize}
\item either $\tan(2\pi\theta_3)\in\Q\cup\{ \infty\}$ and the orbit of every point of $\Tc_{\theta_3}$ is periodic;
\item or  $\tan(2\pi\theta_3)\notin\Q\cup\{ \infty\}$ and the orbit of every point of $\Tc_{\theta_3}$ is dense in $\Tc_{\theta_3}$.
\end{itemize}
Let us assume that $f:\T^3\times\R^3\rightarrow \T^3\times \R^3$ is a symplectic  diffeomorphism such that $f(\Lc)=\Lc$. Then $Df(\R X (j(\theta)))=\R X(f\circ j(\theta))$ because $f$ is symplectic. Hence the image by $f$ of every orbit for $(g_t)$ is another orbit for $(g_t)$, and even the image of a periodic orbit is a periodic orbit, the image of a non-periodic orbit is a non-periodic orbit. Hence the image of an irrational leaf (that is the closure of a non-periodic orbit of $(\psi_t)$) is  an irrational leaf, and by using a limit,  the image of a rational leaf is a rational leaf. Hence there exists an homeomorphism $h:\T\rightarrow \T$ such that:
$$\forall \theta_3\in\T, f(\Lc_{\theta_3})=\Lc_{h(\theta_3)}.$$
and then there exists a continuous   family $(f_{\theta_3})_{\theta_3\in\T}$ of diffeomorphisms $ f_{\theta_3}:\T^2\rightarrow\T^2$ such that:
$$\forall (\theta_1, \theta_2, \theta_3)\in\T^3, f(j(\theta_1, \theta_2, \theta_3))=j(f_{\theta_3}(\theta_1, \theta_2), h(\theta_3)).$$
In other terms, $f_{|\Lc}$ is conjugated to $F:\T^3\rightarrow \T^3$ defined by $$F(\theta_1, \theta_2, \theta_3)= (f_{\theta_3}(\theta_1, \theta_2), h(\theta_3)).$$Then $(f_{\theta_3})_{\theta_3\in\T}$ is an isotopy, hence the action of $f_{\theta_3}$ on $H^1(\T^2, \R)$ is independent of $\theta_3$, and can be represented by a matrix $A\in SL(2, \Z)$. Let us consider the rational leaf $\Lc_0$. Then $\Tc_0$ is foliated by periodic orbits for $(g_t)$ and their homology class is $\begin{pmatrix} 1\\0\end{pmatrix}$. Then $\Tc_{h^n(0)}=F^n(\Tc_0)$ is foliated by periodic orbits whose homology class is $A^n\begin{pmatrix} 1\\0\end{pmatrix}$.%\\

There are two cases:
\begin{enumerate}
\item[$\bullet$]  either the sequence$\left( A^n\begin{pmatrix} 1\\0\end{pmatrix}\right)_{n\in\N} $ is bounded. Because all these homology classes are in the lattice $\Z^2$, there are  only a  finite number of possible values for the terms of this sequence, hence  there exists $N\geq 1$ such that $A^N \begin{pmatrix}1\\ 0\end{pmatrix}= \begin{pmatrix}1\\ 0\end{pmatrix}$. Then $F^N(\Tc_0)$ is foliated by periodic orbits for $(g_t)$ whose homology class is $\begin{pmatrix} 1\\0\end{pmatrix}$, and then   $F^N(\Tc_0)=\Tc_0$: this implies that $F$, and then $f$ cannot be minimal;
\item[$\bullet$] or the sequence$\left( A^n\begin{pmatrix} 1\\0\end{pmatrix}\right)_{n\in\N}$ is unbounded. As $A\in SL(2, \Z)$, this implies that $A$ is parabolic or hyperbolic and that $\displaystyle{\lim_{n\rightarrow +\infty} \left\|  A^n\begin{pmatrix} 1\\0\end{pmatrix}\right\| =+\infty}$. Moreover, the sequence  $\left( A^n\begin{pmatrix} 1\\0\end{pmatrix}\right)_{n\in\N}$ converges in the projective sense. Let $\pm\begin{pmatrix} \cos2\pi \alpha\\ \sin 2\pi\alpha \end{pmatrix}$ be the two vectors with norm 1 that represent this  projective limit (they are eigenvectors for $A$). As  $  A^n\begin{pmatrix} 1\\0\end{pmatrix} $ is a multiple of\break $\begin{pmatrix}\cos(2\pi h^n(0))\\ \sin(2\pi h^n(0))\end{pmatrix}$, we deduce that the sequence $(h^n(0))_{n\in\N}$ has at most two limit points: $\alpha$ and $\alpha+\frac{1}{2}$. This implies that $F$, and then $f$ cannot be minimal.
 \end{enumerate}

%\bigskip

The manifold  that we built just before is a graph. Changing the embedding $j$, we obtain a new submanifolf $\Lc=j(\T^3)$ that is invariant by the geodesic flow of the flat metric, such that the new characteristic flow $(\psi_t)$ is conjugated to the old one. Hence by a similar argument to the previous one, we obtain that the new manifolds satisfies the conclusions of proposition \ref{P1}.

Let us explain how we build the new $\Lc$. Let $\eta:\T\rightarrow \T$ be a smooth map with degree 1 that is not an homeomorphim. We define:
$$j(\theta_1, \theta_2, \theta_3)= (\theta_1, \theta_2,\eta( \theta_3); \cos(2\pi\theta_3), \sin(2\pi\theta_3), 0)
.$$
Then in this case $\Lc=j(\T^3)$ is homotopic to the zero section but is not a graph.

If we choose for $\eta:\T\rightarrow \T$   a smooth map with degree 0, then $\Lc=j(\T^3)$ is non homotopic to the zero section.

\section{Proof of theorem \ref{T1} and corollary \ref{C1}}\label{S3}
We assume that $H$ is a Tonelli Hamiltonian defined on $T^*M$, that $\Lc$ is an invariant $d$-dimensional submanifold with no conjugate points, and that there exists two sequences $(t_n)$ and $(s_n)$  tending to $+\infty$ such that the families $(\varphi_{t_n|\Lc})_{n\in\N}$ and $(\varphi_{-s_n|\Lc})_{n\in\N}$ are equilipschitz. This happens for example when $\varphi_T$ is Lipschitz conjugated to some rotation of the torus or the sphere for some $T\not=0$.

\subsection{The Green bundles}\label{ss31}
Let us recall some facts concerning the orbits with no conjugate vectors that are proved in \cite{Ar1}.

Along every locally minimizing orbit, there exists two (non continuous) invariant Lagrangian bundles in $T^*M$, that are transverse to the vertical bundle. They are called the {\em Green bundles} and denoted by $G_-$ and $G_+$. They satisfy the following properties.

%\medskip
If $K\subset T^*M$ is a compact subset with no conjugate points, then the distance between $G_\pm$ and the vertical bundle on $K$ is bounded from below by a strictly positive constant.

%\medskip

Assume that the orbit of $x\in T^*M$ is locally minimizing; let $v\in T_x(T^*M)$;
\begin{itemize}
\item  if $\displaystyle{\liminf_{t\rightarrow +\infty} \| D\varphi_t v\|<+\infty}$ then $v\in G_-$;
\item if $\displaystyle{\liminf_{t\rightarrow -\infty} \| D\varphi_t v\|<+\infty}$ then $v\in G_+$.
\end{itemize}

%\medskip

At every point where $G_-(x)=G_+(x)$, then $G_-$ and $G_+$ are continuous.

\subsection{A Lagrangian covering}
Let us fix a point $x\in \Lc$. We define Bouligand  contingent cone to $\Lc$ at $x$.

\begin{definition}
We write the definition in coordinates, but the definition is independent of the chosen chart. The   {\em contingent cone} to $\Lc$ at $x$ is the cone of $T_x(T^*M)$ denoted by {$C_x\Lc$} whose elements  are the limits~:
$$v=\lim_{ n\rightarrow \infty} \frac{x_n-x}{ t_n} $$
where $(x_n)$  is a sequence  of elements of $\Lc$ converging to $x$ and  $(t_n)$ is a sequence of elements of $\R_+^*$ converging to $0$.%\\
\end{definition}
Because the $(\varphi_{t_n})$ are equilipschitz, then for every $v\in C\Lc$, the sequence $(D\varphi_{t_n}v)$ is bounded and then $v\in G_-$. Using the sequence $(-s_n)$ in a similar way we obtain: $\forall x\in\Lc, C_x\Lc\subset G_-(x)\cap G_+(x)$.
Moreover, as $\Lc$ is a $d$-dimensional  Lipschitz submanifold, the projection of $C_x\Lc$ on some $d$-plane is onto.  Because $G_-(x)$ and $G_+(x)$ are $d$-dimensional, we deduce: $C_x\Lc=G_-(x)=G_+(x)$.%\\

Hence $\Lc$ is differentiable at $x$, its tangent subspace at $x$ is $G_-(x)$ and then $\Lc$ is Lagrangian. Moreover, as $G_-=G_+$ on $\Lc$, $G_-$ is continuous on $\Lc$ and then $\Lc$ is $C^1$.%\\

As $T\Lc=G_-$ is transverse to the vertical, we finally deduce that $\pi$ is a local diffeomorphism at every point of $\Lc$. Moreover, as $\Lc$ is closed, $\pi(\Lc)$ is a  compact and connected submanifold of $M$ that has same dimension as $M$. Because $M$ is compact and connected, we have then $\pi(\Lc)=M$. Hence $\pi_{|\Lc}$ is a covering map of $M$.

\subsection{A theorem due to Arnol'd}

Let us recall Arnol'd following result:
\begin{proposition}
Let $N\subset  T^*\T^d$ be a closed Lagrangian submanifold such that $\pi_{|N}$ is a covering map; then $\pi_{|N}$ is a diffeomorphism.
\end{proposition}
 A proof of it is provided in \cite{He1}. It is well-known that the result is true if you replace $\T^d$ by any closed manifold. To be complete, we give a proof of:

 \begin{proposition}\label{P3}
Let $N\subset  T^*M$ be a closed Lagrangian submanifold such that $\pi_{|N}$ is a covering map; then $\pi_{|N}$ is a diffeomorphism.
\end{proposition}
\begin{proof}[Proof of proposition \ref{P3}]
\begin{lemma}\label{L2}
There exists a closed one form $\Lambda$ on $M$ such that if $F: T^*M\rightarrow T^*M$ is the symplectic diffeomorphism defined by $F(q,p)=(q, p-\Lambda(q))$, then $F(N)$ is  exact Lagrangian.\end{lemma}
\begin{proof}[Proof of lemma \ref{L2}]As $\pi_{|N}$ is a $m$-fold covering map, we can find for every $q\in M$ $m$ sections $s_1, \dots ,s_m$ of $\pi_{|N}$ defined in a neighborhood $U$ of $q$. Then we define $\Lambda$ in $U$ by: $\displaystyle{\Lambda=\frac{1}{m}\sum_{i=1}^ms_i^*(\nu)}$ where $\nu$ is the 1-form of Liouville. Then $\Lambda$ does not depend on the chart we choose and is closed. We define $F$ as in the lemma and prove now that $F(N)$ is exact Lagrangian. Let $\gamma_1:\T\rightarrow N$ be a closed loop in $N$ and let us denote the projected loop by $\Gamma=\pi\circ\gamma_1$.  Let $\gamma_1,\dots ,\gamma_m$ be the $p$ lifted loops of $\Gamma$. Then every $\gamma_i$ is the image of $\gamma_1$ by some Deck transform $D_i:N\rightarrow N$ of the covering map.  But such a Deck transformation being fiber preserving satisfies above each $q_0\in M$ (with the same notation as at the beginning of the proof): $D_i(q, s_1(q))=(q, s_i(q))$ and then $D_i^*\nu=\nu$. We deduce:
$$\int_{\gamma_i}\nu=\int_{\gamma_1}\nu;$$
and then:
$$\int_{\Gamma}\Lambda=\frac{1}{m}\sum_{i=1}^m\int_{\gamma_i}\nu=\int_{\gamma_1}\nu;$$
this gives the lemma.\end{proof}
Then we can assume that $N$ is exact Lagrangian. The proposition is then a consequence of:
\begin{lemma}\label{L3}
Let $N\subset  T^*M$ be a closed exact Lagrangian submanifold such that $\pi_{|N}$ is a covering map; then $\pi_{|N}$ is a diffeomorphism.

\end{lemma}
\begin{proof}[Proof of lemma \ref{L3}]
As $N$ is exact symplectic, the Liouville 1-form $\nu$ admits a primitive $S: N\rightarrow \R$ along $N$ that is a $C^2$ function. Then we define for every $q$ in $M$: $m(q)=\inf\{ S(x); x\in N, \pi(x)=q\}$ and $M(q)=\sup\{ S(x); x\in N, \pi(x)=q\}$.

Then $m$ is semi-concave and $M$ is semi-convex (see for example the appendix of \cite{Be1} for then definitions and properties of semi-concave functions).  Let $q_0$ be a point where $M-m$ is maximal.  Then  $M$ and $m$ are differentiable at $q_0$ and have the same derivative $\Dc$. %\\

Let us assume that $\pi_{|N}$ is $m$-fold with $m\geq 2$. There is a neighborhood $U$ of $q_0$ in $M$and $m$ sections $s_1, \dots , s_m$ above $U$ that define $N\cap \pi^{-1}(U)$.  Hence there exists $i\not=j$ such that $M(q_0)=S(s_j(q_0))$ and $m(q_0)=S(s_i(q_0))$; let us notice that we have by definition of $S$: $d(S\circ s_i)(q_0)=s_i(q_0)\not= s_j(q_0)=d(S\circ s_j)(q_0)$.%\\

We have then: $S(s_j(q_0))=M(q_0)$ and $\forall q\in U, S(s_j(q))\leq M(q)$. The two functions being differentiable at $q_0$, we deduce: $d(S\circ s_j)(q_0)=dM(q_0)=\Dc$. By a similar argument, we obtain: $d(S\circ s_i)(q_0)=dm(q_0)=\Dc$. This contradicts $d(S\circ s_j)(q_0)\not= d(S\circ s_i)(q_0)$.
\end{proof}\end{proof}
\section{Proof of theorem \ref{T2}, corollaries \ref{C4} and   \ref{C3}}
\subsection{Proof of theorem \ref{T2} and corollary \ref{C4}}
\subsubsection{Proof of theorem \ref{T2}}
 Let $M$ be a closed $n$-dimensional manifold and let $N\subset T^*M$ be a closed $n$-dimensional $C^1$ submanifold with no conjugate points  that is invariant by some Tonelli flow $(\varphi_t)$. We assume that for a dense subset $D\subset N$ then:
 $$\forall x\in D, \forall v\in T_xN, \min\{ \liminf_{t\rightarrow +\infty}\| D\varphi_tv\|, \liminf_{t\rightarrow -\infty} \| D\varphi_tv\|\}<+\infty.$$
 By the dynamical criterion that we recalled in subsection \ref{ss31}, this implies that for every $x\in D$, we have: $T_xN\subset G_-(x)\cup G_+(x)$ and then that $T_xN\subset G_-(x)$ or $T_xN\subset G_+(x)$ because $T_xN$ is a linear space. Hence for every $x$ in $D$, $T_xN$ is Lagrangian and  the distance between $T_xN$ and the vertical at $x$ is bounded from below by a constant  that does not depend on $x\in D$. As $T_xN$ continuously depend on $x$, we deduce that $T_xN$ is Lagrangian and transverse to the vertical bundle at every point of $N$. We then conclude as in section \ref{S3}.
   \subsubsection{Proof of corollary \ref{C4}}
    Let $M$ be a closed $n$-dimensional manifold and let $N\subset T^*M$ be a closed $n$-dimensional $C^1$ submanifold with no conjugate points  that is invariant by some Tonelli flow $(\varphi_t)$. We assume that there exists a compact $C^1$ Lie group $G$ that acts on $N$ and that contains the diffeomorphisms  $ \varphi_{t|N} $ for $t\in\R$.  Then for every $x\in N$, the set $\{ D\varphi_t(x)_{|N}; t\in\R\}$ is compact and we can apply theorem \ref{T2}.
\subsection{Proof of corollary  \ref{C3}}

 Let $M$ be a closed $n$-dimensional manifold and let $N\subset T^*M$ be a closed $n$-dimensional $C^1$ submanifold with no conjugate points  that is invariant by some Tonelli flow $(\varphi_t)$. We assume that $(\varphi_{t|N})$ is such that:
   \begin{enumerate}
   \item[$\bullet$] the   non-wandering set for $(\varphi_{t|N})$ is a finite union of periodic orbits or critical points;
   \item[$\bullet$] all the periodic orbits of $(\varphi_{t|N})$ are non-degenerate in the following sense: if $\tau$ is the period of such a period point, the multiplicity of $1$  as an eigenvalue   of $D\varphi_{\tau|TN}$ is   one and $-1$ is not an eigenvalue; for the critical points, we assume the Hamiltonian is Morse at the critical points contained in $N$.
\end{enumerate}

  Let us study the periodic orbits. If $x$ is a periodic orbit with period $\tau>0$,  we know that $-1$ is not an eigenvalue of $D\varphi_\tau (x)$ and that $1$ is a simple eigenvalue, that corresponds to the flow direction. Let us prove that $D\varphi_\tau(x)$ has no eigenvalue with modulus 1 except $1$. If $e^{i\theta}$ is such an eigenvalue, then the exists a 2-plane $\Pc$ in $T_x(T^*M)$  that is invariant by $D\varphi_\tau(x)$ and such that $D\varphi_\tau(x)_{|\Pc}$ is symplectic and conjugated to a rotation. By the dynamical criterion (see subsection \ref{ss31}), this implies that $\Pc\subset G_-(x)$. Because the restriction of the symplectic form to $\Pc$ is non-zero,  this contradicts the fact that $G_-(x)$ is Lagrangian. Hence the periodic orbits that are in $N$ are hyperbolic. A similar argument implies that the critical points are hyperbolic. We denote by $\Oc_1, \dots , \Oc_m$ the periodic orbits (eventually critical) that are contained in $N$ and by $W^u(\Oc_i, (\varphi_{t|N}))$ and $W^s(\Oc_i, (\varphi_{t|N}))$ their stable and unstable manifolds.

  Because the non-wandering  set of $(\varphi_{t|N})$ is $\Oc_1\cup\dots \cup\Oc_m$, then $$\displaystyle{N=\bigcup_{1\leq i, j\leq n} W^s (\Oc_j,(\varphi_{t|N}))\cap W^u(\Oc_i,(\varphi_{t|N}))}.$$
  If $\Oc_i$ is not an attractive orbit for$(\varphi_{t|N})$ then $W^s(x_i,(\varphi_{t|N})$ is an immersed manifold whose dimension is less that $n$ and then has zero Lebesgue measure. We deduce that there is a dense set $D$ in $N$ such that for all $x\in D$, $\varphi_t(x)$ tends to a repulsive periodic orbit when $t$ tends to $-\infty$ and tends to an attractive periodic orbit when $t$ tends to $+\infty$.

 Let us consider $x\in D$.% \\

We assume that $(\varphi_tx)$ tends to a critical attractive fixed point $x_0$ when $t$ tends to $+\infty$. We can choose $k\in ]0, 1[$ and a Riemannian metric such that in a neighborhood $\Vc$ of $x_0$: $\|D\varphi_{1|N}(y)\|\leq k$. If $t\geq T$ is great enough, $\varphi_tx$ belongs to $\Vc$ and $\| D\varphi_1(\varphi_tx)\|\leq k$. We deduce:
$$\forall n\in \N, \| D\varphi_{T+n}(x)\|\leq \| D\varphi_T(x)\|\prod_{i=0}^{n-1}\| D\varphi_{1|N}(\varphi_{T+i}x)\|\leq \| D\varphi_{T|N}(\varphi_T(x))\|k^n;$$
hence the sequence $(D\varphi_{T+n}(x))_{n\in\N}$ is bounded. %\\

If  $(\varphi_tx)$ tends to a true attractive periodic orbit $\Oc$, then $\Oc$ is a normally hyperbolic (attractive) submanifold for $(\varphi_{t|N})$. Then there exist $x_0\in \Oc$ such that $x\in W^s(x_0)$ (see for example \cite{HPS}). Any vector of $T_xN$ can be written as the sum of $\lambda X(x)$ where $X$ is the Hamiltonian vector field and a vector tangent $v$ to $W^s(x_0)$. Then $D\varphi_t(x)X(x)=X(\varphi_tx)$ is bounded and $D\varphi_t(x)v$ tends to $0$ when $t$ tends to $+\infty$. Finally, the family  $(D\varphi_t(x))_{t>0}$ is bounded.

We have then proved that we can apply theorem \ref{T2}.

\section{Appendix}\label{s4}
We will prove the more or less classical result that we mentioned in the introduction:
\begin{proposition}\label{P4}
Let $\Lc\subset M$ be a closed Lagrangian submanifold of a compact symplectic manifold $(M,\omega)$. Let $f:\Lc\rightarrow \Lc$ be a $C^1$ diffeomorphism that is isotopic to identity. Then there exists a symplectic $C^1$ diffeomorphism $F:M\rightarrow M$ such that $F_{|\Lc}=f$.
\end{proposition}
 \noindent{\bf Proof of proposition \ref{P4}} Because $f$ is isotopic to identity, we may write $f=f_n\circ f_{n-1}\circ \dots \circ f_1$ where the $f_n$ are diffeomorphisms that are $C^1$-close to identity. Hence we just need to prove proposition \ref{P4} for $f$ $C^1$-close to $Id_\Lc$.

 By \cite{Wei}, there exist a neighborhood $U$ of $\Lc$ in $M$, a neighborhood $V$ of $\Lc$ in $T^*\Lc$ and a symplectic diffeomorphism $\Phi:U\rightarrow V$ such that $\phi_{|\Lc}={\rm Id}_\Lc$.

 Let $q_0\in\Lc$ be any point of $\Lc$. We can choose a symplectic chart $(V_{q_0}, \phi_{q_0}=(q,p))$ at $q_0$ in $V$  such that $V_{q_0}\subset V$ and $\Lc\cap V_{q_0}=\{ p=0\}$. From now, we work in this chart and then when we write $f(q)$ this means ``$f(q)$ in chart''. Moreover, we can extend $f$ to $\R^d$ into $\tilde f$ that is $C^1$ close to identity in a smooth way. We introduce the following notations:
 $F_p(q)=F(q,p)=\tilde f(q).p$ for every $(q, p)\in \R^d\times \R^d$  and $\alpha:\R\rightarrow \R_+$ is an even function whose support is contained in $[-1, 1]$, constant in a neighborhood of $0$ and such that $\int\alpha=1$. We too use the notations: $\displaystyle{\forall (u_1, \dots, u_d)\in\R^d, \eta_t(u)=\prod_{i=1}^d\frac{1}{|t|}\alpha(\frac{u_i}{t})}$ and $\eta=\eta_1$.

 We define then a function $ A: V_{q_0}\rightarrow \R$ in coordinates in the following way; if $(q,p)\in V_{q_0}$, then
 \begin{enumerate}
 \item[$\bullet$] $A(q, 0)=0$;
 \item[$\bullet$] if $p\not=0$, then $A(q,p)=(\eta_{\| p\|}*F_p)(q)=\int_{\R^d} \eta_{\| p\|}(  q-u  )F_p(u)du$. \end{enumerate}
 \begin{lemma}\label{L4}
 The function $A$ is $C^2$ and $\forall q, \frac{\partial A}{\partial q}(q, 0)=0, \frac{\partial A}{\partial p}(q, 0)=\tilde f(q)$,\\ $\frac{\partial^2 A}{\partial p_i\partial p_j}(q,0)=\frac{\partial^2A}{\partial q_j^2}(q,0)=0$ and $\frac{\partial^2A}{\partial p_j\partial q_i}(q,0)=\frac{\partial f_j}{\partial q_i}(q)$.

 \end{lemma}
 \begin{proof}[Proof of lemma \ref{L4}]
 If $p\not=0$, observe that:
 $$A(q, p)=\frac{1}{\| p\|^d}\int_{\R^d}\eta\left( \frac{  q-u }{\| p\|}\right) (\tilde f(u).p)du=\int_{\R^d}\eta (  v )( \tilde f(q-\| p\|v).p)dv.$$
 Hence: $\displaystyle{\lim_{p\rightarrow 0, q\rightarrow q_0}A(q,p)=0=A(q,0)}$.%\\

 Moreover, if $p\not=0$:
 $$\frac{\partial A}{\partial q}(q,p).\delta q=\int_{\R^d}\eta(v)(D\tilde  f(q-\| p\| v)\delta q).pdv;$$
$$\frac{\partial A}{\partial p}(q,p)=\int_{\R^d}\eta(v)\tilde f(q-\| p\| v)dv-\int_{\R^d}\eta(v)\left((D\tilde  f(q-\| p\| v)v).p\right)\frac{p}{\| p\|}dv;$$
 hence $\displaystyle{\lim_{p\rightarrow 0, q\rightarrow q_0}\frac{\partial A}{\partial q}(q,p)=0}$ and $\displaystyle{\lim_{p\rightarrow 0, q\rightarrow q_0}\frac{\partial A}{\partial p}(q,p)=\tilde f(q_0)}$; we deduce that $A$ is $C^1$ and
 $\frac{\partial A}{\partial p}(q,0)=\tilde f(q)$ and $\frac{\partial A}{\partial q}(q,0)=0$.

 For $p\not=0$, we can write too: $$\frac{\partial A}{\partial q_i}(q,p)=\int_{\R^d}\frac{1}{\| p\|^{d+1}}\frac{\partial\eta}{\partial q_i}(\frac{q-u}{\| p\|})(\tilde f(u).p)du=\left( \int_{\R^d}\frac{\partial\eta}{\partial q_i}(v)\tilde f(q-\| p\| v)dv\right) \frac{p}{\| p\|};$$
 and then:
 $$\frac{\partial^2A}{\partial q_j\partial q_i}(q,p)=\left( \int_{\R^d}\frac{\partial \eta}{\partial q_i}(v)\frac{\partial\tilde f}{\partial q_j}(q-\| p\|v)dv\right)\frac{p}{\| p\|}$$
 admits a continuous extension at $(q,0)$ in:
  $\frac{\partial^2A}{\partial q_j\partial q_i}(q,0)=0$.%\\

  Moreover:$\frac{\partial^2A}{\partial p_j\partial q_i}(q,p)=$
$$ \frac{1}{\| p\|} \int_{\R^d}\frac{\partial\eta}{\partial q_i}(v)\tilde f_j(q-\| p\| v)dv-\int_{\R^d}\frac{\partial \eta}{\partial q_i}(v)(D\tilde f (q-\| p\|v).v)dv \frac{p_jp}{\| p\|^2}$$$$- \int_{\R^d}\frac{\partial\eta}{\partial q_i}(v)\tilde f(q-\| p\| v)dv\frac{p_jp}{\| p\|^3};$$
 Let us recall that
$\int_{\R^d}\frac{\partial \eta}{\partial q_i}(v)v_jdv=0$ for every $j\not=i$ and that $\int_{\R^d}\frac{\partial \eta}{\partial q_i}(v)v_idv=-1$. Let us now study the different terms of this sum:
\begin{enumerate}
\item $ \displaystyle{ \frac{1}{\| p\|} \int_{\R^d}\frac{\partial\eta}{\partial q_i}(v)\tilde f_j(q-\| p\| v)dv=   \int_{\R^d}\frac{\partial\eta}{\partial q_i}(v)\frac{\tilde f_j(q-\| p\| v)-\tilde f_j(q)}{\| p\|}dv}
$ is equal to:\break%\\
$\displaystyle{-\int_{\R^d}\frac{\partial\eta}{\partial q_i}(v)  \left(D\tilde f_j(q).v+\varepsilon_q (\|p\| v) \right) dv
}$ where $\varepsilon_q$ designates a function that uniformly (in $q$) tends to $0$ at $0$. We have then:
$$\lim_{\| p\|\rightarrow 0}\frac{1}{\| p\|} \int_{\R^d}\frac{\partial\eta}{\partial q_i}(v)\tilde f_j(q-\| p\| v)dv=
-\int_{\R^d}\frac{\partial\eta}{\partial q_i}(v)   D\tilde f_j(q).v dv=\frac{\partial \tilde f_j}{\partial q_i}(q).$$
\item$\displaystyle{-\int_{\R^d}\frac{\partial \eta}{\partial q_i}(v)(D\tilde f (q-\| p\|v).v)dv \frac{p_jp}{\| p\|^2}}$\\$\displaystyle{= -\int_{\R^d}\frac{\partial \eta}{\partial q_i}(v)\left ((D\tilde f (q).v) +\beta_p(\| p\| v) \right)dv  \frac{p_jp}{\| p\|^2}}$ where $\beta_q$ designates a function that uniformly (in $q$) tends to $0$ at $0$.%\\

 \item $\displaystyle{- \int_{\R^d}\frac{\partial\eta}{\partial q_i}(v)\tilde f(q-\| p\| v)dv\frac{p_jp}{\| p\|^3}=  \int_{\R^d}\frac{\partial\eta}{\partial q_i}(v)(D\tilde f(q)v+\alpha_q(\| p\| v))dv\frac{p_jp}{\| p\|^2}}$\break
where $\alpha_q$ designates a function that uniformly (in $q$) tends to $0$ at $0$.%\\
\end{enumerate}
We have finally proved that:
$$\lim_{p\rightarrow 0}\frac{\partial^2A}{\partial p_j\partial q_i}(q,p)=\frac{\partial \tilde f_j}{\partial q_i}(q).$$

 In a similar way, we can rewrite the derivatives with respect to $p$: $ \frac{\partial A}{\partial p_j}(q,p)=$
 $$
-\int_{\R^d}\left(d.\eta\left( \frac{q-u}{\| p\|}\right)+D\eta\left( \frac{q-u}{\| p\|}\right) \frac{q-u}{\| p\|}\right)(\tilde f(u).\frac{p_jp}{\| p\|^{d+2}})du$$$$+\frac{1}{\| p\|^d}\int_{\R^d}\eta(\frac{q-u}{\| p\|})\tilde f_j(u)du$$
 that is equal to:
 $$-\int_{\R^d}\left( d.\eta(v)+D\eta (v).v\right)(\tilde f(q-\| p\|v).\frac{p_jp}{\| p\|^2})dv + \int_{\R^d}\eta (v)\tilde f_j(q-\| p\|v)dv.$$
 Let us recall that $d.\int_{\R^d}\eta(v)dv=-\int_{\R^d}D\eta(v)vdv=d$.%\\
 We have:$\frac{\partial^2A}{\partial q_ip_j}(q,p)=$
 $$-\int_{\R^d}\left( d.\eta(v)+D\eta (v).v\right)(\frac{\partial \tilde f}{\partial q_i}(q-\| p\|v).\frac{p_jp}{\| p\|^2})dv+ \int_{\R^d}\eta (v)\frac{\partial \tilde f_j}{\partial q_i}(q-\| p\|v)dv;$$
 Using the fact that $$-\int_{\R^d}\left( d.\eta(v)+D\eta (v).v\right)\frac{\partial \tilde f}{\partial q_i}(q )dv=0,$$ we deduce that
 $\displaystyle{\lim_{p\rightarrow 0}\frac{\partial^2 A}{\partial q_i\partial p_j}(q,p)=\frac{\partial\tilde  f_j}{\partial q_i}(q)}$.

 Moreover,
 \begin{equation*}\begin{split} \frac{\partial^2A}{\partial p_j^2} (q,p )=&-\int_{\R^d}\eta(v)(D\tilde f_j(q-\| p\|v).v)\frac{p_j}{\| p\|}dv \\ & +
  \int_{\R^d}(d.\eta (v)+ D\eta (v).v)( \tilde f(q-\| p\| v)\frac{(2p_j^2-\| p\|^2)p}{\| p\|^4} \\ & +( D\tilde f(q-\| p\| v).v )\frac{p_j^2p}{\| p\|^3}-\tilde f_j(q-\| p\| v)\frac{p_j}{\| p\|^2})dv.\end{split}\end{equation*}
 The  integral  $-\int_{\R^d}\eta(v)(D\tilde f_j(q-\| p\|v).v) dv$ has for limit $-\int_{\R^d}\eta(v)(D\tilde f_j(q).v )dv=0$ then the limit of the first integral in the sum is zero. For the second integral, we use the fact that $\tilde f(q-\| p\| v)-\tilde f(q)=-(D\tilde f(q)\| p\| v+\|p\| \| v\| \varepsilon_q(\| p\|v)$ and $D\tilde f(q-\| p\| v)=D\tilde f(q)+\alpha_q(\| p\| v)$ where $\varepsilon_q$ and $\alpha_q$ tend (uniformly in $q$) to $0$ at $0$. We then obtain:
 $$\lim_{p\rightarrow 0}\frac{\partial^2A}{\partial p_j^2}(q,p)=0.$$
 In a similar way, we obtain: $\displaystyle{\lim_{p\rightarrow 0}\frac{\partial^2A}{\partial p_j\partial p_i}(q,p)=0}.$%\\

 Finally, $A$ is $C^2$.\end{proof}
 We have finally built an function $A$ that is, in some neighbourhood of $q_0$, $C^2$-close to the function $(q,p)\mapsto p.q$ (because $f$ is close to identity). Using a smooth partition  of unity, it is easy to build a $C^2$ function $\Ac$ whose support is in $V$, that is $C^2$-close to $(p, q)\mapsto 0$ in the considered charts and such that
 $\frac{\partial \Ac}{\partial q}(q, 0)=0$ and $\frac{\partial \Ac}{\partial p}(q, 0)=(f(q)-q)$. Then $\Ac$ is the generating function of a $C^1$-symplectic diffeomrphism $F$ of $T^*\Lc$ that is identity outside $V$ and is defined in $V$ by: $F(q,p)=(Q,P)$ if and only if $Q-q=\frac{\partial\Ac}{\partial P}(q, P)$ and $p-P=\frac{\partial \Ac}{\partial q}(q,P)$. Using $\Phi$ and the fact that $F_{|(\R^d\times \R^d)\backslash V}=Id_V$, we built $G: M\rightarrow M$ symplectic that coincides with $\Phi^{-1}\circ F\circ \Phi$ in $U$ and $Id$ in $M\backslash U$. Then we have $G_{|\Lc}=f$.

%For acknowledgements section, please don't number the section, please begin it with \section*{Acknowledgements}
\section*{Acknowledgments}   I am  grateful  to J.-C.~Alvarez-Paiva, T.~Barbot and C.~Vite\-rbo for stimulating discussions, and to P.~Bolle for explaining to me the construction of the surface  in corollary \ref{Co0}. I am grateful to the anonymous referee too.

% You may incorporate your references as follows in your main tex file.
% Using BibTex is not recommended but can be handled.

\medskip
% The data information below will be filled by AIMS editorial staff
Received March 2013; revised July 2013.
\medskip

\end{document}